\documentclass[10pt,a4paper,reqno]{amsart}
\usepackage[latin1]{inputenc}
\usepackage{amsmath,enumerate,amsthm,graphicx,color,verbatim}
\usepackage{amsfonts,amscd}
\usepackage{latexsym}
\usepackage{amssymb}
\newtheorem{thm}{Theorem}

\newtheorem{claim}{Claim}

\newtheorem{lemma}{Lemma}
\newtheorem{cor}{Corollary}

\newcommand{\Z}{\mathbb{Z}}
\newcommand{\D}{\mathbb{D}}

\newcommand{\Q}{\mathbb{Q}}

\newcommand{\SL}{\mathrm{SL}}
\newcommand{\C}{\mathbb{C}}

\newcommand{\T}{\mathbb{T}}

\newcommand{\im}{\mathrm{Im}}
\newcommand{\re}{\mathrm{Re}}

\newcommand{\norm}{\vert\vert}
\begin{document}
\title[H\"older continuity for spectral measures of extended CMV]{The H\"older continuity of spectral measures of an extended CMV matrix}
\author[Munger and Ong]{Paul E. Munger and Darren C. Ong}
\begin{abstract}We prove results about the H\"older continuity of the spectral measures of the extended CMV matrix, given power law bounds of the solution of the eigenvalue equation. We thus arrive at a unitary analogue of the results of Damanik, Killip and Lenz about the spectral measure of the discrete Schr\"odinger operator.
\end{abstract}

\thanks{D.O. was supported in part by NSF grant DMS--1067988}
\maketitle
\begin{section}{Introduction}
The CMV matrix is a central object in the study of orthogonal polynomials on the unit circle (OPUC), where it plays a role analogous to that of the Jacobi matrix in the study of orthogonal polynomials on the real line (OPRL). More precisely, given a probability measure on the unit circle, we can perform a Gram-Schmidt orthogonalization process on $\{1,z,z^2,\ldots\}$ using the standard $L^2$ inner product to obtain a sequence of orthogonal polynomials $\{\varphi_0(z), \varphi_1(z),\varphi_2(z),\ldots\}$ that obey a recursion relation known as the Szeg\H o recursion, given by $\varphi_0(z)=1$ and 
\[
\varphi_{n+1}(z)=\frac{z\varphi_n(z)-\overline{\alpha(n)}\varphi_n^{*,n}(z)}{\rho(n)}.
\]
 Here $\{\alpha(n)\}_{n=0}^\infty\in \mathbb D^\infty$ are the recursion coefficients (also known as \em Verblunsky coefficients\em), and $\rho(n)=\sqrt{1-\vert \alpha(n)\vert^2}$. The ${}^{*,n}$ operator is defined by $P^{*,n}(z)=z^n \overline {P(1/\overline{z})}$. 

The Szeg\H o recursion can be written in a matrix form as follows:
\begin{equation}\label{e.tmbasic1}
\begin{pmatrix} \varphi_{n+1}(z) \\ \varphi_{n+1}^*(z) \end{pmatrix} = \frac{1}{\rho(n)} \left( \begin{array}{c
c} z & - \overline{\alpha(n)} \\ - \alpha(n) z & 1 \end{array} \right) \begin{pmatrix} \varphi_{n}(z) \\ \varphi
_{n}^*(z) \end{pmatrix}.
\end{equation} 
 
The (one-sided) CMV matrix $\mathcal C$ is a unitary operator on $\ell^2(\mathbb Z_{\geq 0})$ given by 

\begin{equation}\label{CMVmatrix-1}
\left(
\begin{array}{cccccc}

\overline{\alpha(0)}&\overline{ \alpha(1)}\rho(0)&\rho(1)\rho(0)&0&0&\ldots\\
\rho(0)&- \overline{\alpha(1)}\alpha(0)&-\rho(1)\alpha(0)&0&0&\ldots\\
0& \overline{\alpha(2)}\rho(1)&-\overline{\alpha(2)}\alpha(1)&\overline{\alpha(3)}\rho(2)&\rho(3)\rho(2)&\ldots\\
0& \rho(2)\rho(1)&-\rho(2)\alpha(1)&-\overline{\alpha(3)}\alpha(2)&-\rho(3)\alpha(2)&\ldots\\
0& 0&0&\overline{\alpha(4)}\rho(3)&-\overline{\alpha(4)}\alpha(3)&\ldots\\
\ldots& \ldots&\ldots&\ldots&\ldots&\ldots\\
\end{array}
\right).
\end{equation}

There are numerous important connections between the matrix, the polynomials, and the probability measure on the unit circle. For instance, the spectral measure of $\mathcal C$ is exactly the probability measure from which the orthogonal polynomials are derived. The CMV matrix thus enables us to use spectral theoretic tools in OPUC, and enables us to understand OPUC as a unitary analogue of OPRL. We recommend \cite{Simon} as a good reference for OPUC and the CMV matrix.

Now let us redefine $\{\alpha(n)\}_{n=-\infty}^\infty$ as a two-sided infinite sequence, so we can introduce the two-sided CMV matrix $\mathcal E$,
\begin{equation}\label{CMVmatrix-2}
\left(
\begin{array}{ccccccc}
\ldots&\ldots& \ldots&\ldots&\ldots&\ldots&\ldots\\
\ldots&-\overline{\alpha(0)}\alpha(-1)&\overline{ \alpha(1)}\rho(0)&\rho(1)\rho(0)&0&0&\ldots\\
\ldots&-\rho(0)\alpha(-1)&- \overline{\alpha(1)}\alpha(0)&-\rho(1)\alpha(0)&0&0&\ldots\\
\ldots&0& \overline{\alpha(2)}\rho(1)&-\overline{\alpha(2)}\alpha(1)&\overline{\alpha(3)}\rho(2)&\rho(3)\rho(2)&\ldots\\
\ldots&0& \rho(2)\rho(1)&-\rho(2)\alpha(1)&-\overline{\alpha(3)}\alpha(2)&-\rho(3)\alpha(2)&\ldots\\
\ldots&0& 0&0&\overline{\alpha(4)}\rho(3)&-\overline{\alpha(4)}\alpha(3)&\ldots\\
\ldots&\ldots& \ldots&\ldots&\ldots&\ldots&\ldots\\
\end{array}
\right).
\end{equation}
Here $\mathcal E$ is a unitary operator on $\ell^2(\mathbb Z)$. While $\mathcal E$ has a somewhat looser connection with the orthogonal polynomials since the set of recursion coefficients is now two-sided infinite, it is the more natural object of study when we choose to define the $\alpha(n)$ dynamically (please refer to sections 10.5-10.16 of \cite{Simon} for a background on this point of view). 

Furthermore, the two-sided CMV matrix has emerged as a useful tool in the study of quantum walks, a connection first explored in \cite{CMVG}. For certain models of quantum walks, where $\psi\in\ell^2(\Z)$ refers to the initial state of a particle, $\mathcal E^k\psi$ gives us the state after a quantum walk of $k$ steps. 

There is, naturally, a way to relate the measures of an extended CMV matrix and the measures of the constituent halves. This formula was discovered by Gesztesy and Zinchenko, and will be explained in the next section. As an application of this formula we present an OPUC analogue of \cite{DKL}. That is, we derive theorems about the H\"older continuity of the spectral measures of $\mathcal E$ given power law bounds on the entries of formal eigenvectors of one of the corresponding $\mathcal C$ ``halves".

In \cite{IRT}, \cite{DL}, and \cite{DKL} power law bounds on formal eigenvectors are established for OPRL with a Sturmian sequence of recursion coefficients.  By \emph{Sturmian}, we mean that $\alpha(n) = v(n)\alpha + (1-v(n))\beta$, where $\alpha$, $\beta$ (the \emph{alphabet}) are complex numbers of modulus less than one, and $v(n) = [(n+1)\omega]-[n\omega]$. Here $\omega \in (0,1) \setminus \Q$ is the \emph{frequency}. In Section 4, we apply analogous methods to extended CMV matrices with $\omega$ equal to the golden mean.

This is the second of a trilogy of loosely related papers concerned with the dynamical spreading of the time-independent quantum walk model whose coins arise from a Fibonacci sequence. More precisely, the coin at the $n$th site is rotation by $\theta_n$, where $\theta$ is any suffix of a Fibonacci sequence (see \cite{Simon} 12.8):
$$\dots \theta_a, \theta_b, \theta_a, \theta_a, \theta_b, \dots. $$

The first of the trilogy is \cite{DMY}, which establishes that the formal solutions $u$ of $\mathcal Eu=zu$ obey certain power law bounds (where $\mathcal{E}$ is the transition matrix that describes the quantum walk above). This paper shows that $u$ obeying power law bounds implies H\"older continuity of spectral measures of $\mathcal E$. Lastly, \cite{DFV} asserts that for unitary operators, H\"older continuity of spectral measure implies dynamical spreading of the corresponding quantum walk.

We wish here to thank David Damanik and Fritz Gesztesy for many helpful suggestions and comments.
\end{section}

\begin{section}{Main tools and background}
In this section, we will explain an essential formula proven in \cite{GZ}. We require a way to relate the extended CMV operator $\mathcal E$ with the two one-sided CMV operators that comprise its two halves. More precisely, if we modify $\alpha(-1)=-1$, then (\ref{CMVmatrix-2}) becomes the direct sum of operators on $\ell^2([0,\infty)\cap\Z)$ and $\ell^2([-1,-\infty)\cap\Z)$ of the form (\ref{CMVmatrix-1}). We label the halves as $\mathcal C_+$ and $\mathcal C_-$ respectively. We also note that henceforth in this paper $\mathcal E$ refers to the unmodified extended CMV matrix.

First, let us label $F_+(z)$, the Carath\'eodory function corresponding to $\mathcal C_+$, and $F_-(z)$, the Carath\'eodory function corresponding to $\mathcal C_-$. Carath\'eodory functions are holomorphic maps from $\mathbb D$ to the right half plane $\{z\vert \mathrm{Re   }\hspace{5pt}  z>0\}$.  We also say a function is anti-Carath\'eodory when its negative is Carath\'eodory. The correspondences between a given CMV matrix and its Carath\'eodory function are explored more fully in Section 1.3 of \cite{Simon}. Briefly, a Carath\'eodory function is the CMV analogue of the $m$-function in the theory of Jacobi matrices, and is connected to the spectral theory of the CMV matrix.

For example, where $c_i$ are the moments of the spectral measure of the one-sided CMV matrix $\mathcal C$, its Carath\'eodory function $F$ may be expressed as $F(z)=1+2\sum_{n=1}^\infty c_n z^n$. It is also true that $\mathrm Re F(re^{i\theta}) d\theta/2\pi$ converges weakly to the spectral measure of $\mathcal C$ as $r\to 1$. Finally, we note that where $\mu$ is the spectral measure of a CMV matrix, its Carath\'eodory function is given by the formula 
\[F(z)=\int \frac{e^{i\theta}+z}{e^{i\theta}-z}d\mu(\theta).\]

The Green's function (or resolvent function) for $\mathcal{E}$ is computed using formal eigenvalues to $\mathcal{C}_\pm$ and $\mathcal{C}^T_\pm$. 
\begin{lemma}[Lemma 3.1 in \cite{GZ}]

Let $z\in \mathbb C\setminus(\partial \D\cup\{0\})$, and let $M_-$ be an anti-Caratheodory function in (\cite{GZ}, Lemma 2.20), which is, by (2.139) in \cite{GZ} related to $F_-$ by
$$M_-(z) = \frac{\re(1-\overline\alpha_0) - i \im(1+\overline\alpha_0)F_-(z)}{i\im(1-\overline\alpha_0) - \re(1+\overline\alpha_0)F_-(z)}.$$
 Let $u_{\pm}$ be $\ell^2$ solutions to $(\mathcal C_\pm-z)u=0$, and let $v_{\pm}$ be $\ell^2$ solutions to $(\mathcal C^T_\pm-z)v=0$, normalized by
$$v_-(z,0) = -1 + M_-(z),\ v_+(z,0) = -1+F_+(z),$$
$$u_-(z,0) = z+zM_-(z),\ u_+(z,0) = z+zF_+(z).$$
We may extend these solutions to solutions of $(\mathcal E-z)w=0$ and $({\mathcal E}^T-z)w=0$.

Then the resolvent function $(\mathcal E-z)^{-1}(x,y)$ can be expressed as

\begin{equation}\label{GZ}
\frac{-1}{2z(F_+(z)-M_-(z))}
\begin{cases}
u_-(z,x)v_+(z,y) &\text{ if $x<y$ or $x=y$ odd, }\\
u_+(z,x)v_-(z,y) &\text{ if $x>y$ or $x=y$ even, }
\end{cases}
\end{equation}
\end{lemma}

\end{section}
\begin{section}{Applications of the formula}
For a $\ell^2(\mathbb N)$-vector $u$, and a positive integer $n$, we define $\norm u\norm_n$ as $\sqrt{\sum_{j=0}^n \vert u(j)\vert^2}$. We can also define $\norm u\norm_k$, for $k$ positive but not an integer as a linear interpolation of $\norm u\norm_n$.

\begin{lemma}\label{JLI}
Suppose, for a one-sided CMV matrix $\mathcal  C$, that every solution of 
\begin{equation}\label{e.tmbasic}
\begin{pmatrix} \eta_{n+2}(z) \\ \eta_{n+1}(z) \end{pmatrix} = \frac{1}{\rho(n)} \left( \begin{array}{c
c} z & - \overline{\alpha(n)} \\ - \alpha(n) z & 1 \end{array} \right)\begin{pmatrix} \eta_{n}(z) \\ \eta_{n-1}(z) \end{pmatrix},
\end{equation}
with $\vert \eta_0(z)\vert^2+\vert \eta_1(z)\vert^2=2$ obeys the estimate
\begin{equation}\label{estimate}C_1L^{\gamma_1}\leq \norm \eta(z)\norm _L\leq C_2 L^{\gamma_2},
\end{equation}
for $L>0$ sufficiently large. Then
\begin{equation}\label{boundarycondition}
\sup_{\lambda\in \partial\D} \left\vert \frac{(1-\lambda)+(1+\lambda)F(rz)}{(1+\lambda)+(1-\lambda)F(rz)}\right\vert\leq C_3(1-r)^{\beta-1},
\end{equation}
where $\beta=2\gamma_1/(\gamma_1+\gamma_2)$.
\end{lemma}

\begin{proof}
This is a consequence of the Jitomirskaya-Last inequality for OPUC (see \cite{Simon} Section 10.8), which says that 
$$\frac{\norm\psi^\lambda(z)\norm_{x(r)}}{\norm\varphi^\lambda(z)\norm_{x(r)}} \lesssim |F^\lambda(rz)| \lesssim  \frac{\norm\psi^\lambda(z)\norm_{x(r)}}{\norm\varphi^\lambda(z)\norm_{x(r)}}.$$ 
Here, $F^\lambda$ is the Carath\'eodory function corresponding to the Alexandrov measure $\mu_\lambda$ (refer to Theorem 3.2.14 of \cite{Simon}). Its first and second kind orthogonal polynomials are $\varphi^\lambda$ and $\psi^\lambda$. The function $x(r)$ is defined by $(1-r)\norm\varphi^\lambda(z)\norm_{x(r)} \norm\psi^\lambda(z)\norm_{x(r)} = \sqrt{2}$. 

The required inequality is equivalent to $|F^\lambda(rz)| \lesssim (1-r)^{\beta -1}$. This is true if $\frac{\norm\psi^\lambda(z)\norm_{x(r)}}{\norm\varphi^\lambda(z)\norm_{x(r)}} \lesssim (1-r)^{\beta - 1}$, by the Jitomirskaya-Last inequality. Because $\varphi^\lambda$ and $\psi^\lambda$ solve $(\eta_{n+2}(z), \eta_{n+1}(z))^T = T_n(z) (\eta_n(z), \eta_{n-1}(z))^T$ with initial conditions $(1,\overline\lambda)$ and $(1, -\overline\lambda)$, the hypothesis applies to $\psi^\lambda$ and $\varphi^\lambda$. Therefore, \[\norm\psi^\lambda(z)\norm_{x(r)}^\beta \norm\varphi^\lambda(z)\norm_{x(r)}^{\beta - 2} \lesssim x(r)^{\gamma_1 (\beta - 2) + \gamma_2 \beta} \simeq 1.\] By the definition of $x(r)$, it follows that \[\norm\psi^\lambda(z)\norm_{x(r)} ^\beta \norm\varphi^\lambda(z)\norm_{x(r)} ^{\beta-2} \lesssim (1-r)^{\beta - 1} \norm\psi^\lambda(z)\norm_{x(r)} ^{\beta-1} \norm\varphi^\lambda(z)\norm_{x(r)} ^{\beta-1},\] which is equivalent to the required inequality. 
\end{proof}

\begin{thm}Given $z \in\Sigma$, suppose that the estimate (\ref{estimate}) holds.  Then, where $G_{kl}(z)=(\delta_k, (\mathcal E-z)^{-1}\delta_l)$, 

\[ \vert G_{00}(rz)+G_{11}(rz)\vert \leq C_4 (1-r)^{\beta-1},\]
for all $r\in (0.9,1)$ and $C_4$ a $z$ and $r$-independent constant. Consequently, $\Lambda$ is $\beta$-H\"older continuous at $z$.

In particular, assume that $S \subset \partial \D$ is a Borel set such that there are constants $\gamma_1,\gamma_2$ and, for each $z \in S$, there are constants $C_1(z), C_2(z)$ so that
$$
C_1(z) L^{\gamma_1} \leq \norm\eta\norm_L \leq C_2(z) L^{\gamma_2}
$$
for every $z \in S$ and for every solution of \eqref{e.tmbasic} that is normalized. Then, the restriction of every spectral measure of $\mathcal{E}$ to $S$ 
is purely $\frac{2\gamma_1}{\gamma_1 + \gamma_2}$-continuous, that is, it gives zero weight to sets of zero $h^{\frac{2\gamma_1}{\gamma_1 + \gamma_2}}$ measure.
 
\end{thm}
\begin{proof}
The following is a maximum modulus principle argument similar to that in \cite{DKL}. Fix $z\in \Sigma$ and $r\in (0.9,1)$. We consider  (\ref{boundarycondition}) and obtain 
\begin{equation}
\sup_{\lambda\in \partial\D} \left\vert \frac{(1-\lambda)+(1+\lambda)F_+(rz)}{(1+\lambda)+(1-\lambda)F_+(rz)}\right\vert\leq C_3(1-r)^{\beta-1}.
\end{equation}
Since $-M_-(z)$ is a Carath\'eodory function it maps to the right half plane, and so the expression $(M_-(rz)+1)/(M_-(rz)-1)$ has modulus less than $1$. Thus by the maximum modulus principle,  
\[
\left\vert \frac{\left(1-\frac{M_-(rz)+1}{M_-(rz)-1}\right)+\left(1+\frac{M_-(rz)+1}{M_-(rz)-1}\right)F_+(rz)}{\left(1+\frac{M_-(rz)+1}{M_-(rz)-1}\right)+\left(1-\frac{M_-(rz)+1}{M_-(rz)-1}\right)F_+(rz)}\right\vert\leq C_3(1-r)^{\beta-1}.
\]
Now if we simplify the expression on the left, we have 
\[ 
\left\vert \frac{1-M_-(rz)F_+(rz)}{F_+(rz)-M_-(rz)}\right\vert.
\]

A table on page 181 of \cite{GZ} computes the values of $u_\pm(1)$, $u_{\pm}(0), v_\pm(1)$, and $v_{\pm}(0)$. We use $k_0=-1$ in that table. Note that we index our solutions $u,v$ differently, so in their notation, 
\begin{align*}
u_+(n)=&q_+(n-1)+F_+p_+(n-1),\\
u_-(n)=&q_+(n-1)+M_-p_+(n-1),\\
v_+(n)=&s_+(n-1)+F_+r_+(n-1),\\
v_-(n)=&s_+(n-1)+M_-r_+(n-1).
\end{align*}
We also label our Verblunsky coefficients differently than they do: their $\alpha_n$ is written as $-\overline{\alpha(n)}$ in our notation. Using these calculations and (\ref{GZ}), we can write $G_{00}+G_{11}$ as

\[-\frac{(-1+F_+)(1+M_-)}{2(F_+-M_-)}-\frac{[z+\overline{\alpha(0)}+M_-(z-\overline{\alpha(0)})][-1-\alpha(0)z+F_+(1-\alpha(0)z)]}{2\rho(0)^2z(F_+-M_-)}.
\]

For $r$ approaching $1$, $\vert G_{00}(rz)+G_{11}(rz)\vert$ gets large when $F_+(rz)-M_-(rz)$ is close to zero, or when $F_+(rz)$ and $F_-(rz)$ both go to infinity. In both these cases, 

\[ 
\vert G_{00}(rz)+G_{11}(rz)\vert\leq C_4\left\vert \frac{1-M_-(rz)F_+(rz)}{F_+(rz)-M_-(rz)}\right\vert,
\]
for an appropriate constant $C_4$.
It is not difficult to see then, as a consequence
\[\left\vert G_{00}(rz)+G_{11}(rz)\right\vert\leq C_4(1-r)^{\beta-1}.
\]
Let us first note the connection between $G_{00}+G_{11}$ and the Carath\'eodory function $F$ corresponding to $\mathcal E$ and $d\Lambda$. We have by definition
\[F(z)=\int\frac{e^{i\theta}+z}{e^{i\theta}-z} d\Lambda(\theta).\]
Let us also define 
\[d\Lambda_r(\theta)=\mathrm{Re} F(re^{i\theta})\frac{d\theta}{2\pi}.\]
It is well known that $d\Lambda_r$ converges to $d\Lambda$ weakly. We note also that

\begin{align*}
F(z)=&\int \frac{e^{i\theta}+z}{e^{i\theta}-z} d\Lambda(\theta)\\
=&1+2z\int \frac{1}{e^{i\theta}-z}d\Lambda(\theta)\\
=&1+2z(G_{00}(z)+G_{11}(z)).
\end{align*}
We then deduce that $\Lambda(z)$ is uniformly $\beta$-H\"older continuous on $\Sigma$. Writing $z=e^{i\Theta}:$
$$\Lambda[e^{i(\Theta-\epsilon)}, e^{i(\Theta+\epsilon)}] = \int _{\Theta-\epsilon} ^{\Theta+\epsilon} 1 d\Lambda(\Theta). $$

We note that for sufficiently small $\epsilon$ the above is less than

$$2\epsilon \left(\mathrm{Re} F((1-\epsilon)z)+1\right) \le C \epsilon ^\beta,$$
since $\beta\leq1$.
\end{proof}

Note that if we let $C_1, C_2$ be $z$-dependent, the theorem still holds, except that $C_3$ is also $z$-dependent. We can then conclude:

\begin{thm}\label{hcont}
Let $\Sigma $ be a Borel subset of $\partial \D$, and let $\mathcal C$ be a CMV operator on $\ell^2(\mathbb N)$. Suppose there are constants $\gamma_1,\gamma_2$ such that for each $z\in\Sigma$, every normalized solution of $\eta(z)$ of the transfer matrix recursion (\ref{e.tmbasic}) obeys the estimate
\[ C_1(z)L^{\gamma_1}\leq \norm \eta\norm _L \leq C_2(z)L^{\gamma_2}\]
for $L>0$ sufficiently large. Let $\beta=2\gamma_1/(\gamma_1+\gamma_2)$. Then any extension $\mathcal E$ of $\mathcal{C}$ to $\ell^2(\Z)$ has purely $\beta$-continuous spectrum on $\Sigma$. Moreover, if $C_1(z)$  and $C_2(z)$ are independent of $z$, then for any $\varphi\in \ell^2$ of compact support, the spectral measure of $(\mathcal E,\varphi)$ is uniformly $\beta$-H\" older continuous on $\Sigma$.
\end{thm}
Before we proceed, we state and prove the following well-known fact for the reader's convenience:
\begin{lemma}
For any $n \in \Z$, $\{\delta_{2n},\delta_{2n+1}\}$ form a spectral basis for $\mathcal E$.
\end{lemma}
\begin{proof}[Proof of lemma]
First, let us show that $\delta_{2n+2}$ is in the span $S_{2n,2n+1}$ of \[\{\mathcal E^k \delta_{2n}\}_{k\in\Z}\cup \{\mathcal E^k \delta_{2n+1}\}_{k\in\Z}.\]
First, note that we have 
\begin{align}
\label{E1}\mathcal E\delta_{2n+1}=&\overline{\alpha(2n+1)}\rho(2n)\delta_{2n}-\overline{\alpha(2n+1)}\alpha(2n)\delta_{2n+1}\\
&+\overline{\alpha(2n+2)}\rho(2n+1)\delta_{2n+2}+\rho(2n+2)\rho(2n+1)\delta_{2n+3}, \nonumber\\
\label{E2}\mathcal E\delta_{2n+2}=&\rho(2n+1)\rho(2n)\delta_{2n}-\rho(2n+1)\alpha(2n)\delta_{2n+1}\\
&-\overline{\alpha(2n+2)}\alpha(2n+1)\delta_{2n+2}-\rho(2n+2)\alpha(2n+1)\delta_{2n+3}.
\nonumber
\end{align}
This gives us
\begin{align*}
\frac{\alpha(2n+1)}{\rho(2n+1)}\mathcal E\delta_{2n+1}+\mathcal E \delta_{2n+2}=& \left(\frac{\vert \alpha(2n+1)\vert^2\rho(2n)}{\rho(2n+1)}+\rho(2n+1)\rho(2n)\right)\delta_{2n}\\
&-\left(\frac{\vert \alpha(2n+1)\vert^2\alpha(2n)}{\rho(2n+1)}+\rho(2n+1)\alpha(2n)\right)\delta_{2n+1},
\end{align*}
and we conclude that $S_{2n,2n+1}$ contains $\mathcal E\delta_{2n+2}$. Applying $\mathcal E^{-1}$ on both sides of the preceding equation shows that it also contains $\delta_{2n+2}$.

By considering the expressions for $\mathcal E\delta_{2n-1}$ and $\mathcal E\delta_{2n}$ instead, we can similarly show that $\mathcal E\delta_{2n-1}$, and hence $\delta_{2n-1}$ lies in $S_{2n,2n+1}$.

Now let us demonstrate that $\delta_{2n+3}$ is in $S_{2n,2n+1}$. We consider (\ref{E1}) and (\ref{E2}) once more, and this time by eliminating the $\delta_{2n},\delta_{2n+1}$ terms we get

This gives us
\begin{align*}
&\mathcal E\delta_{2n+1}-\frac{\overline{\alpha(2n+1)}}{\rho(2n+1)}\mathcal E \delta_{2n+2}\\
&= \left(\overline{\alpha(2n+2)}\rho(2n+1)+\frac{\vert \alpha(2n+1)\vert^2\alpha(2n+1)}{\rho(2n+1)}\right)\delta_{2n+2}\\
&+\left(\rho(2n+2)\rho(2n+1)+\frac{\vert \alpha(2n+1)\vert^2\rho(2n+2)}{\rho(2n+1)}\right)\delta_{2n+3},
\end{align*}

and this demonstrates that $\delta_{2n+3}$ lies in $S_{2n+1,2n+2}$, and hence $S_{2n,2n+1}$.

 We can similarly show that $\delta_{2n-2}$ lies in $S_{2n-1,2n}$ and hence $S_{2n,2n+1}$, by using the expressions for  $\mathcal E\delta_{2n-1}$ and $\mathcal E\delta_{2n}$ and then eliminating the $\delta_{2n}, \delta_{2n+1}$ terms.

We have now shown that $S_{2n,2n+1}$ contains $\{\delta_{2n-2}, \delta_{2n-1}, \delta_{2n+2},\delta_{2n+3}\}$. A simple induction argument now tells us that $S_{2n,2n+1}=\ell^2(\Z)$.
\end{proof}

\begin{paragraph}{\textbf{Remark}}
It is easy to see that $\{\delta_{2n-1},\delta_{2n}\}$ for any $n$ also form a spectral basis.
\end{paragraph}
\begin{proof}[Proof of Theorem \ref{hcont}]From the lemma and its proof, we see that given a $\phi\in \ell^2(\Z)$ with support on $\{-N,\ldots, N+1\}$, there must exist polynomials $P_0, P_1$ of degree not exceeding $N$ such that $P_0(\mathcal E)\delta_0+P_1(\mathcal E)\delta_1=\phi$. This implies that the spectral measure for $\phi$ is bounded by $q(z)d\Lambda(z)$ for some polynomially bounded function $q(z)$. If $C_1, C_2$ are independent of $z$, then, by the corollary $d\Lambda$ is uniformly $\beta-$ H\"older continuous, and this implies that $qd\Lambda$ is also uniformly $\beta$- H\"older continuous. In the case that $C_1, C_2$ depend on $z$, we know that $d\Lambda$ is $\beta$-continuous. Given any $\phi\in \ell^2$, its spectral measure is dominated by $\Lambda$ and so must be $\beta$-continuous as well.
\end{proof}
\end{section}

\begin{section}{Sturmian Verblunsky coefficients}

\begin{thm}\label{ubound-conds}
Given a sequence  $\{A_n:\T \longrightarrow \SL(2,\C)\}_{n=0} ^\infty$, let us write $M_k (z) = \prod_{n=k} ^0 A_n(z) $. Suppose there are sequences $a_n$ and $q_n$ of natural numbers related by  $q_{n+1} = a_{n+1} q_n + q_{n-1}$, such that $M_{q_{n+1}}(z) = M_{q_{n-1}}(z) M_{q_n}(z) ^{a_{n+1}}$. Let $x_n (z) = \mathrm{tr} M_{q_n} (z)$ and $z_n (z) = \mathrm{tr} M_{q_{n-1}}(z) M_{q_n} (z)$, and put $I(z) = x_{n-1} ^2 + x_n^2 + z_n ^2 - x_{n-1}x_n z_n$.

Suppose that:

\begin{enumerate}

\item The function $I$ is independent of $n$.

\item The sequence $a_n$ is of bounded density: $d = \limsup \frac{1}{N} \sum_{n=1} ^N a_n$ is finite.

\item There is a compact set $\Sigma \subset \T$ and a constant $K$ such that $z \in \Sigma$ iff $|x_n (z)| \le K$ or $|z_n(z)| \le K$ for all $n$.

\end{enumerate}
Then for all $z\in\Sigma$, there exist $\gamma_2(z)$ and $C(z)$ independent of $n$ such that
$$\norm M_n (z) \norm \le C(z) n ^{\gamma_2(z)}.$$
\end{thm}

These conditions are sufficient to apply the argument in \cite{IRT}. One obtains
$$C(z) = L^{4d},\ \gamma_2(z) = 4d\log _2 L,$$
where

\begin{align*} 
L=&\max\left(4\max(2,\sup |x_n|, \sup |z_n|), 4 \norm M_1 \norm, 4 \norm M_0 \norm, 4 \norm M_0 M_1 \norm
\right) \\
&\times (4 + 2\max (2, \sup |x_n|, \sup |z_n|)).
\end{align*}
The method used in \cite{DKL} can be applied to show that for some $C'$,
$$\norm \xi \norm _L \le C'(z) n ^{\gamma(z)}$$
for any solution $\xi(z)$ to the transfer matrix recursion. Compactness of $\Sigma$ and continuity of $C$, $\gamma$ yield a $z$-independent bound by taking the maximum.

\begin{thm}\label{lbound-conds}

Let $q_n$ be the convergents of the continued fraction $[a_1, a_2, a_3, \dots]$. Relaxing the bounded density hypothesis to require only that $q_n$ be bounded above by a geometric sequence, such a sequence of maps into $\SL(2,\C)$ satisfies, for all $z\in\Sigma$,
$$\norm \xi(z) \norm_L \ge C_2(z)L^{\gamma_1(z)} $$
for some $C(z)$, $\gamma_1(z)$, and for $L$ large enough.

\end{thm}

\begin{proof}

The method used in \cite{DKL} applies without any significant changes. It only deals with model-independent properties of the transfer matrices. 

\end{proof}

\begin{claim}
In both cases, an extended CMV matrix with Verblunsky coefficients that have a Fibonacci sequence as a suffix furnishes an example of such a sequence of maps. 
\end{claim}

\begin{proof}\label{qpcmv-conds}

Let $\{T_n (z)\}_{n=0} ^\infty$ be the sequence of $n$-step transfer matrices corresponding to the quasiperiodic CMV operator $\mathcal{E}$. Then $\mathrm{det}(T_n(z)) = z^n$, so that $M_n(z) := T_n(z)/z^{n/2}$ is in $\SL(2,\C)$. It is well known (see \cite{Simon} 12.8) that the family $T_n$ obeys a substitution rule of the necessary type; so does the family $M_n$. Because the spectrum of $\mathcal{E}$ is contained in $\T$, $M_n$ and $T_n$ always have the same operator norms. Finally, that the traces $x_n(z)$ obey the required bound is proved in \cite{Simon} 12.8.

The method in \cite{DKL} provides a simple expression for $\gamma_1$.  Put $q_n \le B^n$, and let $C(\alpha, \beta) := \max\{\max_{|z| = 1} 2+\sqrt{8+I(z)}, \frac{4}{\sqrt{1-|\alpha|^2}\sqrt{1-|\beta|^2}}\}$. Then 
$$\gamma_1 = \frac{\log\left(1+\frac{1}{4C(\alpha,\beta)^2)}\right)}{16\log B}. $$
The constant $C(\alpha,\beta)$ occurs because it bounds $|x_n(z)|$ for $z \in \Sigma$. 

\end{proof}

\begin{cor}\label{holder}

With the notation and assumptions above, the spectral measure of such a CMV operator is uniformly $\beta$-H\"older continuous for $\beta = \frac{2\gamma_1}{\gamma_1 + \gamma_2}$.

\end{cor}
\end{section}
\bibliographystyle{alpha}   % this means that the order of references
			    % is determined by the order in which the
			    % \cite and \nocite commands appear
\bibliography{mybib}
\end{document}